\documentclass{article}
\usepackage{amsmath}
\usepackage{amsthm}
\usepackage{amsfonts}
\usepackage{enumerate}
\usepackage[utf8]{inputenc}
\usepackage[T1]{fontenc}
\usepackage{amssymb}
\usepackage{authblk}

\newcommand{\R}{\mathbb{R}}
\newcommand{\re}{\mathrm{Re}}
\newcommand{\Hn}{\mathcal{H}}

\newcommand{\Cn}{\mathbb{C}^{n\times n}}
\newcommand{\C}{\mathbb{C}}

\newcommand{\dg}{\mathrm{diag}}
\newcommand{\co}{\mathrm{conv}}

\newcommand{\uni}{\mathcal{U}}

\newcommand{\tr}{\mathrm{tr}}

\title{The star-shapedness of a generalized numerical range}
\author[1]{Pan-Shun Lau\thanks{panlau@hku.hk}}
\author[1]{Tuen-Wai Ng\thanks{ntw@maths.hku.hk}}
\author[1]{Nam-Kiu Tsing\thanks{nktsing@hku.hk}}
\affil[1]{\small Department of Mathematics, The University of Hong Kong, Pokfulam, Hong Kong}

\begin{document}
\newtheorem{thm}{Theorem}[section]
\newtheorem{lem}[thm]{Lemma}
\newtheorem{coro}[thm]{Corollary}
\newtheorem{defn}[thm]{Definition}
\newtheorem{prop}[thm]{Proposition}
\newtheorem{exam}[thm]{Example}

\maketitle
\hrule
\section*{\small Abstract}

Let $\Hn_n$ be the set of all $n\times n$ Hermitian matrices and $\Hn^m_n$ be the set of all $m$-tuples of $n\times n$ Hermitian matrices. For $A=(A_1,...,A_m)\in\Hn^m_n$ and for any linear map $L:\Hn^m_n\to\R^\ell$, we define the $L$-numerical range of $A$ by
\[
W_L(A):=\{L(U^*A_1U,...,U^*A_mU): U\in\Cn, U^*U=I_n\}.
\]
In this paper, we prove that if $\ell\leq 3$, $n\geq \ell$ and $A_1,...,A_m$ are simultaneously unitarily diagonalizable, then $W_L(A)$ is star-shaped with star center at $L\left(\frac{\tr A_1}{n}I_n,...,\frac{\tr A_m}{n}I_n\right)$.\\[5pt]
{\footnotesize\emph{AMS Classification:}	15A04, 15A60.\\
\emph{Keywords:} generalized numerical range, star-shapedness, joint unitary orbit}
\vspace{10pt}
\hrule

\section{Introduction}
Let $\Cn$ denote the set of all $n\times n$ complex matrices, and $A\in\Cn$. The (classical) numerical range of $A$ is defined by
\[
W(A):=\{x^*Ax:x\in\C^n, x^*x=1\}.
\]
The properties of $W(A)$ were studied extensively in the last few decades and many nice results were obtained; see \cite{GR,HJ}. The most beautiful result is probably the Toeplitz-Hausdorff Theorem which affirmed the convexity of $W(A)$; see \cite{Hausdorff,Toeplitz}. The generalizations of $W(A)$ remain an active research area in the field. 

For any $A\in\Cn$, write $A=A_1+iA_2$ where $A_1,A_2$ are Hermitian matrices. Then by regarding $\C$ as $\R^2$, one can rewrite $W(A)$ as
\[
W(A):=\{(x^*A_1x,x^*A_2x):x\in\C^n, x^*x=1\}.
\]
This expression motivates naturally the generalization of the numerical range to the joint numerical range, which is defined as follows. Let $\Hn_n$ be the set of all $n\times n$ Hermitian matrices and $\Hn^m_n$ be the set of all $m$-tuples of $n\times n$ Hermitian matrices. The joint numerical range of $A=(A_1,...,A_m)\in \Hn^m_n$ is defined as
\[
W(A)=W(A_1,...,A_m):=\{(x^*A_1x,...,x^*A_mx):x\in\C^n, x^*x=1\}.
\]
It has been shown that for $m\leq 3$ and $n\geq m$, the joint numerical range is always convex \cite{AYP}. This result generalizes the Toeplitz-Hausdorff Theorem. However, the convexity of the joint numerical range fails to hold in general for $m>3$, see \cite{AYP,GJK,LP98}. 

When a new generalization of numerical range is introduced, people are always interested in its convexity. Unfortunately, this nice property fails to hold in some generalizations. However, another property, namely star-shapedness, holds in some generalizations; see \cite{CheungT,Tsing}. Therefore, the star-shapedness is the next consideration when the generalized numerical ranges fail to be convex. A set $M$ is called star-shaped with respect to a star-center $x_0\in M$ if for any $0\leq \alpha\leq 1$ and $x\in M$, we have $\alpha x + (1-\alpha) x_0 \in M$.  In \cite{LP2011}, Li and Poon showed that for a given $m$, the joint numerical range $W(A_1,...,A_m)$ is star-shaped if $n$ is sufficiently large. 

Let $\mathcal{U}_n$ be the set of all $n\times n$ unitary matrices. For $C\in\Hn_n$ and $A=(A_1,...,A_m)\in \Hn^m_n$, the joint $C$-numerical range of $A$ is defined by
\[
W_C(A):=\{(\tr(CU^*A_1U),...,\tr(CU^*A_mU): U\in\uni_n\},
\]
where $\tr(\cdot)$ is the trace function. When $C$ is the diagonal matrix with diagonal elements $1,0,...,0$, then $W_C(A)$ reduces to $W(A)$. Hence the joint $C$-numerical range is a generalization of the joint numerical range. In \cite{AYT2}, Au-Yeung and Tsing generalized the convexity result of the joint numerical range to the joint $C$-numerical range by showing that $W_C(A)$ is always convex if $m\leq 3$ and $n\geq m$. However $W_C(A)$ fails to be convex in general if $m>3$. One may consult \cite{CN} and \cite{CLP} for the study of the convexity of $W_C(A)$. The star-shapedness of $W_C(A)$ remains unclear for $m>3$.

For $A=(A_1,...,A_m)\in\Hn_n^m$, we define the joint unitary orbit of $A$ by 
\[
\uni_n(A):=\{(U^*A_1U,...,U^*A_mU):U \in \mathcal{U}_n\}.\]
For $C\in \Hn_n$, we consider the linear map $L_C:\Hn^m_n \to \R^m$ defined by
\[
L_C(X_1,...,X_m)=(\tr(CX_1),...,\tr(CX_m)).
\]
Then the joint $C$-numerical range of $A$ is the linear image of $\mathcal{U}_n(A)$ under $L_C$. Inspired by this alternative expression, we consider the following generalized numerical range of $A\in\Hn_n^m$. For $A=(A_1,...,A_m)\in \Hn^m_n$ and linear map $L:\Hn^m_n \to \R^\ell$, we define
\[
W_L(A)=L(\mathcal{U}_n(A)):=\{L(U^*A_1U,...,U^*A_mU):U \in \mathcal{U}_n\},
\] and call it the $L$-numerical range of $A$, due to \cite{CT}. Because $L_C$ is a special case of general linear maps $L$, the $L$-numerical range generalizes the joint $C$-numerical range and hence the classical numerical range. 

In this paper, We shall study in Section two an inclusion relation of the $L$-numerical range of $m$-tuples of simultaneously unitarily diagonalizable Hermitian matrices and linear maps $L:\Hn_n^m\to \R^\ell$ with $\ell=2,3$. This inclusion relation will be applied in Section three to show that the $L$-numerical ranges of $A$ under our consideration are star-shaped. 

\section{An Inclusion Relation for $L$-numerical Ranges}

The following results follow easily from the the definition of the $L$-numerical range.
\begin{lem}\label{basic}
Let $(A_1,...,A_m)\in\Hn_n^m$ and $L:\Hn_n^m\to\R^\ell$ be linear. Then the followings hold:
\begin{enumerate}[(i)]
	\item $W_L(\alpha(A_1,...,A_m)+\beta(I_n,...,I_n))=\alpha W_L(A_1,...,A_m)+\beta L(I_n,...,I_n)$ if $\alpha,\beta\in\R$;
	\item $W_L(U^*A_1U,...,U^*A_mU)=W_L(A_1,...,A_m)$ for all unitary $U$.
\end{enumerate}
\end{lem}

In the following we shall consider those $A_1,...,A_m$ which are simultaneously unitarily diagonalizable, i.e., there exists $U\in\uni_n$ such that $U^*A_1U,...,U^*A_mU$ are all diagonal. Hence by Lemma \ref{basic}, we assume without loss of generality that $A_1,...,A_m$ are (real) diagonal matrices. For $d=(d_1,...,d_n)^T\in\R^n$, we denote by $\dg(d)$ the $n\times n$ diagonal matrix with diagonal elements $d_1,...,d_n$. We first introduce a special class of matrices which is useful in studying the generalized numerical range; see \cite{GS,Poon,Tsing}.

An $n\times n$ real matrix $P=(p_{ij})$ is called a pinching matrix if for some $1\leq s< t\leq n$ and $0\leq \alpha\leq 1$,
\[
p_{ij}=\left\{\begin{array}{cc} \alpha, & \text{ if }(i,j)=(s,s) \text{ or } (t,t),\\
1-\alpha, &\text{ if }(i,j)=(s,t) \text{ or } (t,s),\\
1,& \text{ if }i=j\neq s,t ,\\
0& \text{otherwise}.
\end{array}\right.
\]

\begin{defn}
Assume $D=(\dg(d^{(1)}),...,\dg(d^{(m)}))$, {$\hat{D}=(\dg(\hat{d}^{(1)}),...,$ $\dg(\hat{d}^{(m)}))$} where $d^{(1)},...,d^{(m)},\hat{d}^{(1)},...\hat{d}^{(m)}\in\R^n$. We say $\hat{D}\prec D$ if there exist a finite number of pinching matrices $P_1,...,P_k$ such that $\hat{d}^{(i)}=P_1P_2\cdots P_k d^{(i)}$ for all $i=1,...,m$.
\end{defn}
 
The following inclusion relation is the main result in this section.

\begin{thm}\label{inclusion_r1}
Let $D,\hat{D}\in \Hn^m_n$ and $n>2$. If $\hat{D}\prec D$, then for any linear map $L:\Hn^m_n\to \mathbb{R}^3$, we have $W_L(\hat{D})\subset W_L(D)$.
\end{thm}

To prove Theorem \ref{inclusion_r1}, we need some lemmas. For $\theta,\phi\in\R$, let $T_{\theta,\phi}\in\mathcal{U}_n$ be defined by
\[T_{\theta,\phi}=\begin{pmatrix}\cos\theta & \sin\theta e^{\sqrt{-1}\phi} &0\\
-\sin\theta & \cos\theta e^{\sqrt{-1}\phi}& 0 \\ 0 & 0 & I_{n-2}\end{pmatrix}.\]

\begin{lem}\label{ellipse1}
	Let $D=(D_1,...,D_m)\in \Hn^m_n$ be an $m$-tuple of diagonal matrices. Then for any linear map $L:\Hn^m_n\to \mathbb{R}^3$ and $U\in\mathcal{U}_n$, the set of points
	\[
		E_L(D,U):=\{L(U^*T^*_{\theta,\phi}D_1T_{\theta,\phi}U,...,U^*T^*_{\theta,\phi}D_mT_{\theta,\phi}U):\theta\in[0,\pi],~\phi\in[0,2\pi]\}
	\]
	forms an ellipsoid in $\R^3$.
\end{lem}
\begin{proof}
Note that for any $L:\Hn^m_n\to \mathbb{R}^3$, we can always express $L$ as 
	\[
	L(X_1,...,X_m)=\left(\tr\left(\sum^m_{i=1}P_iX_i\right),\tr\left(\sum^m_{i=1}Q_iX_i\right),\tr\left(\sum^m_{i=1}R_iX_i\right)\right)
	\]
for some suitable $P_i,Q_i,R_i\in \Hn_n$, $i=1,...,m$. For $U\in\uni_n$, we write $UP_iU^*=(p^{(i)}_{jk})$, $UQ_iU^*=(q^{(i)}_{jk})$, $UR_iU^*=(r^{(i)}_{jk})$ and $D_i=\dg(d^{(i)}_1,...,d^{(i)}_n)$, $i=1,...,m$. By direct computations, the first coordinate of points in $E_L(D,U)$ is 
	\[
		\begin{aligned}
		&\;\tr\left(\sum^m_{i=1}P_iU^*T^*_{\theta,\phi}D_i T_{\theta,\phi}U\right)\\
		=& \;\tr\left(\sum^m_{i=1}D_iT_{\theta,\phi}UP_iU^*T^*_{\theta,\phi}\right)\\
		=&\; \frac{1}{2}\sum^m_{i=1}(d^{(i)}_1+d^{(i)}_2)(p^{(i)}_{11}+p^{(i)}_{22})+\sum^m_{i=1}\sum^n_{j=3} d^{(i)}_j p^{(i)}_{jj}\\ 
	  & \;\;\;\;\;\;\;+ \frac{1}{2}\sum^m_{i=1}(d_1^{(i)}-d_2^{(i)})(p_{11}^{(i)}-p_{22}^{(i)})\cos 2\theta\\
		& \;\;\;\;\;\;\;+\sum^m_{i=1}(d_1^{(i)}-d_2^{(i)})\re(p_{21}^{(i)}e^{\sqrt{-1}\phi})\sin 2\theta.
		\end{aligned}
	\]
Similarly for the second and the third coordinates of points in $E_L(D,U)$. Note that for $a_1,a_2,b_1,b_2,c_1,c_2\in\R$ and $a_3,b_3,c_3\in\C$, the points $(a_1,b_1,c_1)+(a_2,b_2,c_2)\cos 2\theta+\re(a_3e^{\sqrt{-1}\phi},b_3e^{\sqrt{-1}\phi},c_3e^{\sqrt{-1}\phi})\sin 2\theta$ form an ellipsoid in $\R^3$ when $\theta,\phi$ run through $[0,\pi]$ and $[0,2\pi]$ respectively. Hence $E_L(D,U)$ is an ellipsoid in $\R^3$.
\end{proof}

Note that $E_L(D,U)\subset W_L(D)$ for any $U\in\uni_n$.

\begin{lem}\label{ellipsoid_deg}
Let $D\in \Hn^m_n$ be an $m$-tuple of diagonal matrices with $n>2$. Then for any linear map $L:\Hn^m_n\to \mathbb{R}^3$, there exists $V\in\mathcal{U}_n$ such that $E_L(D,V)$ defined in Lemma \ref{ellipse1} degenerates (i.e., $E_L(D,V)$ is contained in a plane in $\R^3$).
\end{lem}
\begin{proof}
Following the notations in Lemma \ref{ellipse1} and its proof, we let $\alpha_i=d^{(i)}_1-d^{(i)}_2$ for $i=1,...,m$ and $P'=\sum_{i=1}^m \alpha_i P_i\in \Hn_n$. Since $n>2$, by generalized interlacing inequalities for eigenvalues of Hermitian matrices (see \cite{Fan}), there exist $V\in\uni_n$ and $\alpha\in\R$ such that $VP'V^*$ has $\alpha I_2$ as leading $2\times 2$ principal submatrix. For any matrix $M$, let $M_{ij}$ denote its $(i,j)$ entry. Then by taking $U=V$ in the proof of Lemma \ref{ellipse1}, the first coordinate of points in $E_L(D,V)$ is $a+b\cos 2\theta +c\sin 2\theta$ where 
\[\begin{split}
a&=\frac{1}{2}\sum^m_{i=1}(d_1^{(i)}+d_2^{(i)})(p^{(i)}_{11}+p^{(i)}_{22})+\sum_{i=1}^m\sum_{j=3}^n d^{(i)}_jp_{ii}\\
b&=\frac{1}{2} \sum_{i=1}^m \alpha_i \left[ (VP_iV^*)_{11}-(VP_iV^*)_{22}\right]\\
&=\frac{1}{2} \left(V\left(\sum_{i=1}^m \alpha_i P_i\right)V^*\right)_{11}-\frac{1}{2}\left(V\left(\sum_{i=1}^m \alpha_i P_i\right)V^*\right)_{22}\\
&=\frac{1}{2}(VP'V^*)_{11}-\frac{1}{2}(VP'V^*)_{22}\\&=\frac{1}{2}\alpha -\frac{1}{2}\alpha =0,\\
c&=\sum^m_{i=1}\alpha_i\re\left((VP_iV^*)_{21}e^{\sqrt{-1}\phi}\right)\\
&=\re\left[\left(V\left(\sum^m_{i=1}\alpha_i P_i\right)V^*\right)_{21}e^{\sqrt{-1}\phi}\right]\\
&=\re((VP'V^*)_{21}e^{\sqrt{-1}\phi})=0.
\end{split}
\]

Since the first coordinate of points in $E_L(D,V)$ is constant for $\theta\in [0,\pi]$ and $\phi\in [0,2\pi]$, $E_L(D,V)$ degenerates.
\end{proof}

\begin{proof}[Proof of Theorem \ref{inclusion_r1}.]
Let $D=(D_1,...,D_m)=(\dg(d^{(1)}),...,\dg(d^{(m)}))$ and $\hat{D}=(\hat{D}_1,...,\hat{D}_m)=(\dg(\hat{d}^{(1)}), ...,$ $\dg(\hat{d}^{(m)}))$ where $d^{(1)},..,d^{(m)},\hat{d}^{(1)},...,\hat{d}^{(m)}\in\R^n$. We may further assume without loss of generality that $\hat{d}^{(i)}=Pd^{(i)}$ for all $i=1,...,m$ and $P=\begin{pmatrix}\alpha & 1-\alpha \\ 1-\alpha & \alpha\end{pmatrix}\oplus I_{n-2}$ with $0\leq \alpha\leq 1$. Then we have
\[
\hat{D}_i=\alpha T_{0,0}^*D_iT_{0,0}+(1-\alpha)T_{\frac{\pi}{2},0}^*D_iT_{\frac{\pi}{2},0}, \;\;\;i=1,...,m. 
\]
For any $U\in\uni_n$, we have $L(U^*\hat{D} U)\in\co(E_L(D,U))$ where $\co(\cdot)$ denotes the convex hull. By path-connectedness of $\mathcal{U}_n$, there exists a continuous function $f:[0,1]\to \mathcal{U}_n$
such that $f(0)=U$ and $f(1)=V$ where $V$ is defined in Lemma \ref{ellipsoid_deg} and hence $E(D,f(1))$ degenerates.
By continuity, there exists $t\in[0,1]$ such that $L(U^*\hat{D}U)\in E(D,f(t))\subset W_L(D)$.
\end{proof}

Using similar techniques, one can prove that Theorem \ref{inclusion_r1} stills holds for all linear maps $L:\Hn_n^m\to \R^2$ with $n\geq 2$. However, the following example shows that the inclusion relation in Theorem \ref{inclusion_r1} fails to hold if $L:\Hn^m_n\to\R^\ell$ is linear with $\ell>3$.

\begin{exam}
Let $n\geq 2$, $d=(1,...,0)^T$, $\hat{d}=(\frac{1}{2},\frac{1}{2},0,...,0)^T\in\R^n$ and let $O_k$ be the $k\times k$ zero matrix. Consider $D=(\dg(d),O_n,...,O_n)$, $\hat{D}=(\dg(\hat{d}),O_n,...,O_n)\in\Hn_n^m$ and $L:\Hn_n^m\to \R^\ell$ with $\ell\geq 4$ defined by
\[
L(X_1,...,X_m)=(\tr(PX_1),\tr(QX_1),\tr(RX_1),\tr(SX_1),0,...,0)
\]
where
\[
P= \begin{pmatrix}1& 0\\ 0&1\end{pmatrix}\oplus O_{n-2},\;\;\;\;Q= \begin{pmatrix}0& i\\ -i&0\end{pmatrix}\oplus O_{n-2},\]
\[
R= \begin{pmatrix}1& 0\\ 0&-1\end{pmatrix}\oplus O_{n-2},\;\;\;\;S= \begin{pmatrix}0& 1\\ 1&0\end{pmatrix}\oplus O_{n-2}.
\]
Then we have $\hat{D}\prec D$ and $(1,0,...,0)\in W_L(\hat{D})$, but $(1,0,...,0)\notin W_L(D)$.
\end{exam}

\section{Star-shapedness of the $L$-numerical range}

The $L$-numerical range may fail to be convex for linear maps $L:\Hn_n^m\to\R^\ell$ with $\ell\geq 2$ even when $A_1,...,A_m\in\Hn_n$ are simultaneously unitarily diagonalizable; see \cite{AYT}. However, we shall show in this section that for $n>2$, $W_L(A_1,...,A_m)$ is always star-shaped for all linear maps $L:\Hn_n^m\to\R^3$ and simultaneously unitarily diagonalizable $A_1,...,A_m\in\Hn_n$. The following result is the essential element in our proof. 

\begin{prop}\label{ss_pinching}\cite{Tsing}
Let $\mathbb{P}_n$ be the set of all finite products of $n\times n$ pinching matrices. Then for $0\leq\alpha\leq 1$, $\alpha I_n+ (1-\alpha) J_n$ is in the closure of $\mathbb{P}_n$ where $J_n$ is the $n\times n$ matrix with all entries equal $1/n$.
\end{prop}

Note that for any $A\in\Hn_n^m$, $\uni_n(A)$ is compact. Hence  $W_L(A)$ is compact for all linear maps $L$.

\begin{thm}\label{main_thm_R3}
Let $D=(D_1,...,D_m)\in \Hn^m_n$ be an $m$-tuple of diagonal matrices with $n>2$. Then for any linear map $L:\Hn^m_n\to \mathbb{R}^3$ , $W_L(D)$ is star-shaped with respect to star-center $L(\frac{\tr D_1}{n}I_n,...,\frac{\tr D_m}{n}I_n)$.
\end{thm}
\begin{proof}
By Lemma \ref{basic}, we may assume without loss of generality that $\tr D_i=0$ for $i=1,...,m$; otherwise we replace $D_i$ by $D_i-\frac{\tr D_i}{n}I_n$. Let $D_i=\dg(d^{(i)})$ where $d^{(i)}\in\R^n$, $i=1,...,m$. For any $0\leq\alpha\leq 1$, we have $\alpha d^{(i)}= [\alpha I_n+ (1-\alpha) J_n ] d^{(i)} $. Then for any $U\in\uni_n$, by Proposition \ref{ss_pinching}, Theorem \ref{inclusion_r1} and the compactness of $W_L(D)$, we have $\alpha L(U^*DU)\in W_L(\alpha D)\subset \overline{W_L(D)}=W_L(D)$ where $\overline{M}$ denotes the closure of $M$.
\end{proof}

For a linear map $L:\Hn^m_n\to \mathbb{R}^2$, by regarding it as a projection of some linear map $\hat{L}:\Hn^m_n\to \mathbb{R}^3$, we deduce the following corollary easily.

\begin{coro}\label{main_thm_R2}
Let $D=(D_1,...,D_m)\in \Hn^m_n$ be an $m$-tuple of diagonal matrices with $n\geq 2$. Then for any linear map $L:H^m_n\to \mathbb{R}^2$ , $W_L(D)$ is star-shaped with respect to star-center $L(\frac{\tr D_1}{n}I_n,...,\frac{\tr D_m}{n}I_n)$. 
\end{coro}
\begin{proof} We only need to consider the case $n=2$. We may assume without loss of generality that $m=1$ and $D=\dg(1,-1)$. For any linear map $L:\Hn_2\to\R^2$, we express it as $L(X):=(\tr(PX),\tr(QX))$ for some $P,Q\in\Hn_2$. Then we have
\[\begin{split}
W_L(D)&=\{2(x^*Px,x^*Qx)-(\tr P,\tr Q): x\in\C^n, x^*x=1\}\\
			&=2W(P,Q)-(\tr P,\tr Q),
\end{split}
\] which is convex and contains the origin. This implies that $W_L(D)$ is star-shaped with respect to star-center $L\left( \frac{ {\rm tr}\,D}n I_2 \right)$, which is the origin.
\end{proof}

Note that the star-shapedness of the $L$-numerical range for linear maps $L:\Hn_n^m\to\R^\ell$ with $\ell>3$ remains open in the diagonal case. Moreover, for general cases of $A=(A_1,...,A_m)$ where $A_1,...,A_m$ are not necessarily simultaneously unitarily diagonlizable and $L:\Hn_n^m\to\R^2$ with $m\geq 3$, the star-shapedness of $W_L(A)$ is also unclear. However, by applying a result in \cite{CT}, we can show that $L(\frac{\tr A_1}{n}I_n,...,\frac{\tr A_m}{n}I_n)\in W_L(A_1,...,A_m)$ for all linear maps $L:\Hn_n^m\to\R^2$.

\begin{prop}[\cite{CT}, P. 23.]\label{chan_inclusion}
Let $A_k=(a^{(k)}_{ij})\in \Hn_n$, $k=1,...,m$. For $0\leq \epsilon\leq 1$, define $A_k(\epsilon)$ as
\[
A_k(\epsilon)=
\begin{pmatrix}
a^{(k)}_{11} & \epsilon a^{(k)}_{12} & \cdots & \epsilon a^{(k)}_{1n} \\
\epsilon a^{(k)}_{21} &  a^{(k)}_{22} & \cdots & \epsilon a^{(k)}_{2n} \\
\vdots & \vdots & \ddots & \vdots \\
\epsilon a^{(k)}_{n1} & \epsilon a^{(k)}_{12} & \cdots &  a^{(k)}_{nn} 
\end{pmatrix},
 \;\;\;\; k=1,...,m.
\] Then $W_L(A_1(\epsilon),...,A_m(\epsilon))\subseteq W_L(A_1,...,A_m)$ for any linear map $L:\Hn_n^m\to\R^2$.
\end{prop}

\begin{thm}
Let $A=(A_1,...A_m)\in \Hn_n^m$ and $L:\Hn_n^m\to\R^2$ be linear. Then $L(\frac{\tr A_1}{n}I_n,...,\frac{\tr A_m}{n}I_n)\in W_L(A)$.
\end{thm}
\begin{proof}
Define $A_i(\epsilon)$ as in Proposition \ref{chan_inclusion} and note that $\tr A_i(\epsilon) = \tr A_i$ for $i=1,...,m$. Hence by Corollary \ref{main_thm_R2} and Proposition \ref{chan_inclusion}, we have
\[
L\left(\frac{\tr A_1}{n}I_n,...,\frac{\tr A_m}{n}I_n\right)\in W_L(A_1(0),...,A_m(0))\subseteq W_L(A_1,...A_m).
\]
\end{proof}

\end{document}